\documentclass[12pt]{extarticle} 
\usepackage[top=1in, bottom=1in,left=1in,right=1in]{geometry}

\usepackage{makeidx,amsfonts, amsthm, amssymb, amsmath, graphicx, subfig, bbm, color,epstopdf, wrapfig, caption, listings, float, mathtools, hyperref, url, dsfont}
\usepackage{cleveref}
\allowdisplaybreaks
\usepackage[perpage,symbol*]{footmisc}
\usepackage[usenames,dvipsnames,svgnames,table]{xcolor}
\definecolor{mygreen}{rgb}{0.1, 0.76, 0.24}
\usepackage{microtype}

\usepackage[shortlabels]{enumitem}
\setlist{leftmargin=*,itemsep=0.25\itemsep,parsep=0.35\parsep,topsep=0.25\topsep,partopsep=0.29\partopsep}

\newtheorem{theorem}{Theorem}
\newtheorem{other}{Theorem}

\newtheorem{lemma}[other]{Lemma}

\newtheorem{proposition}[other]{Proposition}
\newtheorem*{remark}{Remarks}

\renewenvironment{proof}[1][\proofname]{ { \it\bfseries #1: }}{\qed}

\newcommand{\cF}{{\mathcal{F}}}

\newcommand{\cN}{{\mathcal{N}}}

\newcommand{\field}[1]{\mathbb{#1}}
\newcommand{\R}{\field{R}}

\newcommand{\E}{\field{E}}

\newcommand{\p}{\field{P}}

\newcommand{\mat}[2][rrrrrrrrrrrrrrrrrrrrrrrrrrrrrrrr]{\left[ \begin{array}{#1} #2 \\ \end{array}\right]}

\numberwithin{equation}{section}
\providecommand{\keywords}[1]
{
  \small	
  \textbf{\textit{Keywords: }} #1
}

\title{\bf  A product-CLT and its application in invariance principle of random projection}
\author{Juntao Duan\footnote{Corresponding author. Georgia Institute of Technology, School of Mathematics,
    \href{mailto:juntaoduan@gmail.com}{juntaoduan@gmail.com}},\quad  
    Ionel Popescu\footnote{ University of Bucharest, Mathematics and Computer Science, Institute of Mathematics of the Romanian Academy, \href{mailto:ioionel@gmail.com }{ioionel@gmail.com}}
    , \quad 
    Fan Zhou \footnote{Baidu Research, \href{mailto:fanzhou@baidu.com}{fanzhou@baidu.com}} }
\begin{document}
	
	\maketitle
	
	\microtypesetup{protrusion=true} 
	
\begin{abstract}
	 Johnson-Lindenstrauss lemma states random projections can be used as a topology preserving embedding technique for fixed vectors. In this paper, we try to understand how random projections affect probabilistic properties of random vectors. In particular we prove the distribution of inner product of two independent random vectors $X, Z \in \field{R}^n$ with i.i.d. entries is preserved by random projection $S:\field{R}^n \to \field{R}^m$. More precisely,
	 \[
	  \sup_t \left| \field{P}(\frac{1}{C_{m,n}} X^TS^TSZ <t) - \field{P}(\frac{1}{\sqrt{n}} X^TZ<t)  \right| \le O\left(\frac{1}{\sqrt{n}}+ \frac{1}{\sqrt{m}} \right) 
	 \]
	 
	 This is achieved by proving a general central limit theorem (product-CLT) for  $\sum_{k=1}^{n} X_k Y_k$, where $\{X_k\}$ is a martingale difference sequence, and  $\{Y_k\}$ has dependency within the sequence. We also obtain the rate of convergence in the spirit of Berry-Esseen theorem.
	 
\end{abstract}
\keywords{Johnson-Lindenstrauss lemma; random projection; Central limit theorem; dependent; invariance; inner product, Berry-Esseen; rate of convergence} 

\section{Introduction}

     Due to the internet boom and computer technology advancement in the last few decades, data collection and storage have been growing exponentially. With 'gold' mining demand on the enormous amount of data reaches to a new level, we are facing many technical challenges in understanding the information we have collected. In many different cases, including text and images, data can be represented as points or vectors in high dimensional space. On one hand, it is very easy to collect more and more information about the object so that the dimensionality grows quickly. On the other hand it is very difficult to analyze and create useful models for high dimensional data due to several reasons including computational difficulty as a result of curse of dimensionality and high noise to signal ratio. It is therefore necessary to reduce the dimensionality of data while preserving relevant structures. 
     
     The celebrated Johnson-Lindenstrauss lemma \cite{johnson1984extensions} states that random projections can be used as a general dimension reduction technique to embed topological structures in high dimensional Euclidean space into a low dimensional space without distorting its topology. 
     Since then random projections has been found very useful in many applications such as signal processing and machine learning.  For example fast Johnson-Lindenstrauss random projections is used to approximate K-nearest neighbors to speed up computation \cite{indyk1998approximate, ailon2009fast}.  Random sketching uses random projection to reduce sample sizes in regression model and low rank matrix approximation \cite{woodruff2014sketching}. Random projected features can be used to create low dimensional base classifiers which are combined as robust ensemble model \cite{cannings2017random}. Practitioners found applications of random projection in privacy and security \cite{liu2005random}.
  Before we begin to state our problem, let us state the Johnson Lindenstrauss lemma \cite{burr2018optimal}.
  
  \begin{lemma}[Johnson and Lindenstrauss]
       Given a set of vectors $\{u_1, \cdots, u_k\}$ in $\R^n$, for any $m \ge 8 \varepsilon^{-2} \log k$, there exists a linear map $A:\R^n \to \R^m$ such that 
       \[
       (1-\varepsilon)\|u_i -u_j\| \le 
       \|Au_i -Au_j\| \le (1+\varepsilon)\|u_i -u_j\|
       \]
  \end{lemma}
  
  Given two fixed vectors $X, Z \in \R^n$,  by Johnson-Lindenstrauss lemma,  we can find a random projections $A: \R^n \to \R^m$ such that the projected distance $\|AX-AZ \|$ has only a small distortion of the original distance $\|X-Z \|$. More precisely, 
\begin{align}\label{eqn:J-L lemma norm error order}
    \left[1- O(\frac{1}{\sqrt{m}}) \right] \|X-Z\|^2 \le \| A(X-Z)\|^2 \le \left[1+ O(\frac{1}{\sqrt{m}}) \right] \|X-Z\|^2
\end{align}
Equivalently, this property can be reformulated as random projections preserves the inner product of two vectors (Equivalence can be obtained by elementary computation and polarization identity). 
Namely given $X, Z$ two vectors in the unit ball of $\R^n$ ($\|X\|\le 1, \|Z\|\le 1$) , then there is a random projection $A:\R^n \to \R^m$ such that 
\begin{align}\label{eqn:J-L lemma inner product}
    | \langle AX, AZ \rangle - \langle X, Z \rangle | \le O(\frac{1}{\sqrt{m}})
\end{align}
For general vectors not in the unit ball, the bound on the right hand side has the norms as a factor \[
| \langle AX, AZ \rangle - \langle X, Z \rangle | \le O(\frac{1}{\sqrt{m}})\|X\| \|Z\|\le O(\frac{1}{\sqrt{m}})(\|X\|^2+ \|Z\|^2) 
\]

The natural extension is to consider random vectors $X, Z$. Then we may ask what random projections do to random vectors? Is there an invariance phenomenon in the distribution sense?  Closeness in distribution usually boils down to the difference of the cumulative distribution function. If we look at inner product, then we will be interested in controlling 
\[
\sup_t \left| \p\left( \langle AX, AZ \rangle < t \right) - \p\left( \langle X, Z \rangle < t \right) \right|
\]
  
  In this work we try to address some of these question. In particular, how distributions of randomly projected random vectors changes. We obtain an invariance principle for independent random vectors very similar to the inner product form of Johnson-Lindenstrauss lemma but extended to the distribution sense. Our contributions in this paper includes:
   \begin{enumerate}
       \item We proved random projections preserves distribution of inner product of independent random vectors with i.i.d. entries.  Roughly speaking, two orthogonal random vectors in high dimension remains orthogonal in the randomly projected lower dimensional space.
       
       \item We also quantitatively characterize the distortion of  distribution introduced by random projection. The error term has a bound at most $O(\frac{1}{\sqrt{m}} + \frac{1}{\sqrt{n}})$. For $m\le n$, this shows the error term is the same order as in Johnson-Lindenstrauss lemma.
       
       
       \item A central limit theorem is established for random variables with dependence structure. At the same time, we obtained its Berry-Esseen type rate of convergence. This alone can be of great interests in many applications involving dependence  structure.
   \end{enumerate}

   The rest of the paper is structured as follows. We first state the main theorems in  \cref{sec:theorems}. Then we prove product-CLT in  \cref{sec: CLT theorem} and obtain the rate of convergence in  \cref{sec: CLT rate theorem}. Along the way, we will discuss some equivalent conditions for product-CLT theorems.  In  \cref{sec:random projection}, we prove the theorems concern invariance principle of random projections.

\subsection{Brief review of CLT for dependent random variables}\label{sec:review CLT}
   Central limit theorem plays an important role in probability and has many real world applications. One pitfall in the classical theory is that we can only deal with independent random variables.  There are many attempts to extend the theory to handle dependent random variables. Hoeffding and Robbins \cite{hoeffding1948} formulated one of the early result which shows CLT still holds for locally dependent sequence. One of the most interesting development is the martingale difference central limit theorem in \cite{brown1971}. In a nutshell, if the conditional variance converges in probability, then a Lindeberg condition implies  CLT for the sequence. 
   	\begin{theorem}[Martingale CLT]\label{thm:martingale CLT}
		Let $\{x_k\}$ be a sequence of martingale differences, $\{\cF_k \}$ be the natural filtration, Let $\E x_k^2=1$, denote $\E[x_k^2|\cF_{k-1}]:=\sigma_k^2$. If the following two conditions hold 
		\begin{enumerate}
		    \item $\frac{1}{n}\sum_k \sigma_k^2 \xrightarrow{\; p \;} 1 $
		    \item Lindeberg condition: $\frac{1}{n}\sum_k \E x_k^2 I(|x_k|>\varepsilon \sqrt{n}) \to 0$ for all $\varepsilon>0$.
		\end{enumerate}
	Then 
	\begin{equation*}
	    \frac{1}{\sqrt{n}} \sum_{k=1}^n x_k \xrightarrow{\;  \;} \cN(0,1)
	\end{equation*}
	\end{theorem}
   
   The exact rate of convergence is obtained by  \cite{bolthausen1982}: with uniformly boundedness condition, the rate of convergence is shown as $O(\frac{\log n}{\sqrt{n}})$. Slightly more general results can be found in \cite{haeusler1988} and \cite{mourrat2013}. There is another line of research considering mixing weak dependence which is extensively discussed in \cite{Bradley2007} and \cite{Dedecker2007}. A mixing condition requires dependence between random variables in the sequence decays as their positions are further apart. Essentially, far apart random variables become almost independent.



\section{Main theorems}\label{sec:theorems}

	\begin{theorem}[product-CLT]\label{thm:product_CLT}
		Given random variables $\{X_k\}$ such that  $\E X_k = 0 $ and $\E X_k^2 = 1$. Given another sequence of random variables $\{Y_k\}$. Assume $\{Y_k\}$ are independent with  $\{X_k\}$  ($Y_k$ and $Y_{k'}$ could be dependent). Assume all third moments exist and bounded, namely there is fixed large number $A$
		\begin{equation}\label{eqn:third order condition}
		    \E [|X_k|^3]<A<\infty, \quad \E [|Y_k|^3]<A<\infty, \quad \forall k
		\end{equation}
        Further assume 
	\begin{equation}\label{eqn:conditional moments}
		    \E[ X_k| \cF_{k-1}]=0, \quad  \E[ X_k^2| \cF_{k-1}]=1
		\end{equation}
		 where $\cF_k$ is the filtration generated by the (martingale difference) sequence $\{X_k\}$. \\
		Assume $\{Y_k\}$  satisfies
		\begin{equation}\label{eqn:LLN Y_k^2} 
		\frac{1}{n}\sum_{k=1}^n Y_k^2 \xrightarrow{\; p \;} 1
		\end{equation}
		Then we have the following CLT
		\begin{equation*}
		  \frac{1}{\sqrt{n}} \sum_{k=1}^n X_k Y_k \to G  
		\end{equation*}
		where $G$ is the standard Gaussian random variable.
	\end{theorem}
\begin{remark}
    The product-CLT can be viewed as an extension of Martingale-CLT \cref{thm:martingale CLT}. If $X$ is a vector of martingale difference sequence which has CLT by \cref{thm:martingale CLT}. Our product-CLT asserts that if there is another $Y$ vector with complicated unknown dependence but satisfies a law of large number condition, then the dot product $X^TY$ has a CLT. This extension is useful because no other CLT can deal with a sequence $\{X_i Y_i\}$ with unknown dependence. As we will see in the proof of invariance principle of random projections, there is no way to apply martingale-CLT directly. Instead, we can decouple the dependence, for example extract a sequence of independent random variables $X$, and  a sequence of $Y$ that has complicated dependence controlled by law of large number on the squares. 
\end{remark}		
	
In principle, one can replace the third order moment condition \cref{eqn:third order condition} by the Lindeberg condition. But we prefer it to keep the argument compact.  
After proving the theorem, we will also give a few conditions that guarantees \cref{eqn:LLN Y_k^2}

Moreover, we are interested in the rate of convergence which will need control of higher order moments. Indeed, in developing a Berry-Esseen type rate of convergence theorem, we will also need assumptions on how fast the average of $\{Y_k^2\}$ converges. We state our result as follows, 
	  \begin{theorem}[Rate of convergence product-CLT]\label{thm:rate of convergence}
    	Assume all conditions, except LLN of $Y_i^2$, in  \cref{thm:product_CLT} holds. Further assume if rate of convergence for LLN of $Y_k^2$ is controlled by the following condition
        \begin{align}\label{eqn:assumption_rate}
           \E \left[ 1 \wedge \left|\sqrt{\frac{1}{n} {\sum_{k=1}^n Y_k^2}}-1 \right| \right] <O(\varepsilon_n)
        \end{align}
    	where $\varepsilon_n$ converges to zero. Then we have 
    	\[
    	\left| \p(\frac{1}{\sqrt{n}}\sum_{k=1}^n X_k Y_k<t) - \p(G<t) \right| \le O(\frac{1}{\sqrt{n}}  \vee \varepsilon_n) \quad \forall t\in \R 
    	\]
    	where $G$ is the standard normal random variable.
  \end{theorem}

We then use the product-CLT theorem to obtain invariance of the distribution of inner product of randomly projected or embedded random vectors.
\begin{theorem}[Random matrix inner product CLT] \label{thm:rand_proj CLT}

Given two independent random vectors in $\R^n$: 
\[
X=\mat[c]{x_1\\ \vdots\\ x_n}, Z=\mat[c]{z_1\\ \vdots\\ z_n}
\]
with i.i.d. entries. And assume $ \E x_i = \E z_i = 0$, $\E x_i^2 = \E z_i^2 = 1$, $\E |x_i|^3 \vee \E |z_i|^3 < C <\infty $. Consider a  random  matrix $S: \R^n \to \R^m$ with i.i.d. entries  and $ \E S_{i,j} = 0$ and  $ \E S_{i,j}^2 = 1$.
Further assume $S,X,Z$ are all independent and $\E S_{1,1}^8 \vee  \E z_1^4 < C<\infty$, then we have 
\[
\frac{1}{\sqrt{m^2n+mn^2}} X^TS^TSZ \to \cN(0,1) \qquad \text{as } m,n \to \infty
\]
\end{theorem}

\begin{theorem}[Random matrix invariance principle] \label{thm:rand_proj rate of CLT}

Given the same moment assumptions as in \cref{thm:rand_proj CLT}, the following bounds hold,
\begin{align}
    \sup_t \left| \p(\frac{1}{\sqrt{m^2n+mn^2}} X^TS^TSZ <t) - \p(G<t)  \right| & \le O\left(\frac{1}{\sqrt{n}}+ \frac{1}{\sqrt{m}} \right) \label{eqn:rand_proj rate to normal} 
    \\
    \sup_t \left| \p(\frac{1}{\sqrt{m^2n+mn^2}} X^TS^TSZ <t) - \p(\frac{1}{\sqrt{n}} X^TZ<t)  \right| &\le O\left(\frac{1}{\sqrt{n}}+ \frac{1}{\sqrt{m}} \right)  \label{eqn:rand_proj rate to original vector}
\end{align}
where $G$ is a standard normal random variable.
\end{theorem}



\bigskip

\section{Proof of  \cref{thm:product_CLT}}\label{sec: CLT theorem}

\subsection{A proof based on Lindeberg swapping}
	\begin{proof}
		Let us begin with the Lindeberg argument.
		
		Take any function $f$ from $ C_c^{\infty}(\R)$ smooth function with bounded support on the real line. Let $S_n=  \sum_{1}^{n} X_i Y_i  $. Let $Z, \{Z_i\}_{1\le i \le n}$ be independent standard normal random variables. It is sufficient to show 
		\[
		\E [f(\frac{1}{\sqrt{n}} S_n)] -\E [f(Z)]  \to 0 
		\]
   		
		
		
		
		

        Our strategy is to split the difference into two parts 
        \[
		\E [f(\frac{1}{\sqrt{n}} S_n)] -\E [f(\frac{1}{\sqrt{n}} \sum_{i=1}^{n}Z_iY_i )],  \quad \text{and } \E [f(\frac{1}{\sqrt{n}} \sum_{i=1}^{n}Z_iY_i )]- \E[f(Z)]
		\]
        then show both are small.
        
		First step, let us try to show
		\[
		\E [f(\frac{1}{\sqrt{n}} S_n)] -\E [f(\frac{1}{\sqrt{n}} \sum_{i=1}^{n}Z_iY_i )]  \to 0 
		\]
		We write the difference as a telescopic sum, 
		\[
		\Delta_n:=f(\frac{1}{\sqrt{n}} S_n ) -f(\frac{1}{\sqrt{n}}\sum_{i=1}^{n} Z_iY_i)= \sum_{k=1}^{n} f(T_k ) -f(T_{k-1} )
		\]
		where 
		\[
		T_k= \frac{1}{\sqrt{n}} \left[ \sum_{i=1}^{k} X_iY_i+ \sum_{i=k+1}^{n}Z_iY_i \right]
		\]
		
		To make notation easier to read, denote 
		\[
		U_k:= \frac{1}{\sqrt{n}} \left[ \sum_{i=1}^{k-1} X_iY_i+ \sum_{i=k+1}^{n}Z_iY_i \right] 
		\]
		It is easy to see $U_k = T_k-\frac{1}{\sqrt{n}} X_kY_k =T_{k-1} -\frac{1}{\sqrt{n}} Z_kY_k$.
		Now let us take a Taylor expansion on $f(T_k)$, $f(T_{k-1})$ around $U_k$, 
		\[
		f(T_k) =f(U_k) + f'(U_k) \frac{1}{\sqrt{n}} X_kY_k +\frac{1}{2}f''(U_k)\frac{1}{n} X_k^2 Y_k^2 +O(n^{-\frac{3}{2}} X_k^3Y_k^3 \sup_x f'''(x))
		\] 
		\[
		f(T_{k-1}) =f(U_k) + f'(U_k) \frac{1}{\sqrt{n}} Z_k Y_k +\frac{1}{2}f''(U_k)\frac{1}{n} Z_k^2 Y_k^2 +O(n^{-\frac{3}{2}} Z_k^3 Y_k^3 \sup_x f'''(x))
		\] 
		Since $Y_k$ is independent with $X_k,  Z_k$, by conditioning on $\cF_{k-1}$ the first order terms match. 
		\begin{align*}
		    \E [f'(U_k) \frac{1}{\sqrt{n}} X_kY_k] 
		    & = \E [f'(U_k) \frac{1}{\sqrt{n}} Y_k\E[X_k|\cF_{k-1}] ]  =0  \\
		    \E [f'(U_k) \frac{1}{\sqrt{n}} Z_kY_k] 
		    & = \E [f'(U_k) \frac{1}{\sqrt{n}} Y_k\E[Z_k] ]  =0  
		\end{align*}
		Similar argument shows second terms match,
		\begin{align*}
		    \E [f''(U_k) \frac{1}{n} X_k^2Y_k^2] 
		    & = \E [\E [f''(U_k) \frac{1}{n} X_k^2Y_k^2 | \cF_{k-1}] ] \\
		    & = \E [\frac{1}{n}f''(U_k)  Y_k^2 \E [ X_k^2 | \cF_{k-1}] ] \\
		    & = \E [\frac{1}{n}f''(U_k)  Y_k^2  ]  \\
		    \E [f''(U_k) \frac{1}{n} Z_k^2Y_k^2] 
		    & = \E [f''(U_k) \frac{1}{n} Y_k^2] \E [Z_k^2] \\
		    & = \E [\frac{1}{n}f''(U_k)  Y_k^2  ] 
		\end{align*}
		Therefore, we obtain 
		\[
		\E f(T_k)-f(T_{k-1}) = O(n^{-\frac{3}{2}} \E X_k^3Y_k^3 \sup_x f'''(x)  )
		\]
		Sum up the $n$ terms,
		\[
		\E \Delta_n = O(\frac{1}{\sqrt{n}} \E (X_k^3 +Z_k^3)Y_k^3 \sup_x f'''(x)  )
		\]
		
		In the case $X_k, Y_k$ have finite third moments, we conclude replacing $X_i$ by Gaussian random variables will only introduce the  difference of the order $n^{-1/2}$
		\[
		\E \Delta_n = O(\frac{1}{\sqrt{n}} )
		\]
		Now it suffices to show \[
		\frac{1}{\sqrt{n}} \sum_{i=1}^{n}Z_iY_i \to \cN(0,1)
		\]
		
		Notice by computing the moment generating function, we can verify, for all $n$
		\[
		\frac{1}{\sqrt{\sum Y_i^2}} \sum_{i=1}^{n}Z_iY_i \sim \cN(0,1)
		\]
		Then by Slutsky's theorem and condition $\frac{1}{n} \sum_{i=1}^{n}Y_i^2 \to 1$, we conclude our desired result.
		
	\end{proof}


\subsection{Alternative assumptions}
	Here we discuss a variation of the condition \cref{eqn:LLN Y_k^2} in product-CLT. This version has the advantage that the assumptions are  easier to verify in practice. We only impose the mixed second moments conditions which can be approximated with empirical data.
	\begin{proposition}\label{prop:second moment condition}
		In  \cref{thm:product_CLT}, if $Y_k$ satisfied,
		\begin{equation*}
		 \E Y_k^2 \to 1, \quad \E [|Y_k|^4]<C < \infty, \qquad  \forall  k\in \mathbb{N}
		\end{equation*}
		
		Further assume the mixed second moments satisfy
		\begin{equation} \label{assumption:mixed second moments}
		\frac{1}{n^2}\sum_{i\ne j}\E [Y_i^2 Y_j^2 ]\xrightarrow[]{\quad n\to \infty \quad} 1
		\end{equation}
		Then the following LLN holds
		\[
		\frac{1}{n}\sum_k Y_k^2 \xrightarrow{\  p \ } 1
		\]
	\end{proposition}
	
	\begin{proof}
	By Chebyshev's inequality,
	\begin{align*}
	    \p (\frac{1}{n}|\sum_i (Y_i^2-1)|>\varepsilon)
	    &\le \frac{\frac{1}{n^2}\E [|\sum_i (Y_i^2-1)|^2]}{\varepsilon^2}
	\end{align*}
	Notice 
	\begin{align*}
	    \frac{1}{n^2}\E [|\sum_i (Y_i^2-1)|^2] 
	    &= \frac{1}{n^2}\left[ \sum_{i} \E Y_i^4 + \sum_{i\ne j} \E Y_i^2 Y_j^2 - 2n\sum_{i} \E Y_i^2 +n^2 \right]\\
	    &\to \frac{1}{n^2}\left[  \sum_{i} \E Y_i^4 +\sum_{i\ne j} \E Y_i^2 Y_j^2 \right] -1 \\
	    & \to 0
	\end{align*}
	Therefore we see for any $\epsilon >0$, 
	\begin{equation*}
	    \p (\frac{1}{n}|\sum_i (Y_i^2-1)|>\varepsilon)
	     \to 0
	\end{equation*}
	This implies there is a weak law of large number for the sequence $Y_n^2$, namely
	\[
	\frac{1}{n}\sum_k Y_k^2 \xrightarrow{\  p \ } 1
	\]
	\end{proof}
	
In practice one will only need to verify that the average 
$\frac{1}{n^2} \sum_{i,j} Y_i^2 Y_j^2 $ is close to 1. It turns out that the mixed second moments condition is equivalent to a fourth moment convergence condition.
\begin{proposition} 
    In  \cref{prop:second moment condition}, the condition \cref{assumption:mixed second moments} is equivalent to 
    \begin{equation}
        \E \left[ P_n^4 \right] \to 3, \qquad \text{where } P_n =\frac{1}{\sqrt{n}} \sum_{i} X_i Y_i 
    \end{equation}
\end{proposition}
\begin{proof}
	\[
	\E[P_n^4]=n^{-2} \sum_{1\le i_1,\dots i_4\leq n} \E X_{i_1}Y_{i_1}\dots  X_{i_4}Y_{i_4}
	\]
	Now we want to analyze the indices $I=\{i_1, i_2, i_3, i_4\}$. If one of index $i_k$ is different from the other three, then  $\E [X_{i_k}|\cF_{k-1}]=0$ implies the whole product vanish. Therefore the only surviving terms must be either all indices the same or indices appear as pairs. Namely 
	\[
	\E[P_n^4]=n^{-2} \left[\sum_{1\le i\le n} \E X_i^4 Y_i^4 +  3 \sum_{1\le i \ne j \le n}  \E X_{i}^2Y_{i}^2 X_{j}^2Y_{j}^2 \right]
	\]
	where the factor $3$ is because the pairs have three cases $\{ (i_1=i_2, i_3=i_4), (i_1=i_3, i_2=i_4), (i_1=i_4, i_2=i_4) \}$.   
	
	Then notice $\E (X_{i}^2X_{j}^2) = \E[\E [X_{i}^2X_{j}^2|\cF_{min(i,j)}]] =1$, we find
	\[
	\E X_{i}^2Y_{i}^2 X_{j}^2Y_{j}^2=\E (X_{i}^2X_{j}^2) \E Y_{i}^2 Y_{j}^2= \E Y_{i}^2 Y_{j}^2
	\]
	
	Combining with the assumption that fourth moment is bounded we see 
	\[
	\E[P_n^4] \to 3 \Longleftrightarrow \frac{1}{n^2}\sum_{i\ne j} \E Y_{i}^2 Y_j^2 \to 1
	\]
\end{proof}





\section{Proof of  \cref{thm:rate of convergence}}\label{sec: CLT rate theorem}
 The proof will be several steps. First we record a variation formula of Gaussian density in  \cref{lemma:Normal-variation}. Then we use the variation formula to rewrite the error term by introducing a standard normal variable in  \cref{lemma:normal_difference}. Then we use Lindeberg type argument to reduce the control of error to control of two terms. One term is a telescopic sum which we will control in  \cref{lemma:part_bound_1} with the moments information. The other term is the difference of two cumulative distribution functions (cdfs) that are close to normal cdfs which we will control in  \cref{lemma:part_bound_2} with the LLN property of $Y_i^2$, namely condition \cref{eqn:assumption_rate}. 
 
\begin{lemma}\label{lemma:Normal-variation}
Let $X$ and $\xi$ be two independent random variables. Let $\sigma=\sqrt{\E \xi^2}$. Let $\Phi$ be the cumulative distribution of standard normal. Denote 
\[
\delta=\sup_t|\p(X\le t)-\Phi(t)| \qquad \delta^*=\sup_t |\p(X+\xi\le t) -\Phi(t)|
\]
Then 
\[
\delta \le 2\delta^* +\frac{5}{\sqrt{2\pi}}\sigma, \quad \delta^* \le 2 \delta + \frac{3}{2\sqrt{\pi}}\sigma
\]
\end{lemma}
\begin{proof}
   See for example \cite{bolthausen1982} 
\end{proof}

\begin{lemma}\label{lemma:normal_difference}
Denote
$$
    \delta := \sup_t\left|\p\left(\frac{\sum X_iY_i}{\sqrt{n}}\le t\right)-\Phi(t)\right|, \quad
    \delta_{\xi} := \sup_t\left|\p\left(\frac{\xi+ \sum X_iY_i }{\sqrt{n}}\le t\right)-\p\left(\frac{\xi }{\sqrt{n}}+G\le t\right) \right|
$$
Given the same setting in  \cref{thm:rate of convergence}, and let $G,\xi$ be independent standard normal random variable. Then
\[
\delta
\le 2 \delta_{\xi} +\frac{3}{\sqrt{n}}
\]

\end{lemma}
\begin{proof}

    By  \cref{lemma:Normal-variation} we have
    \begin{align*}
         \eta: & =\sup_t\left|\p\left(G\le t\right) - \p\left(\frac{\xi }{\sqrt{n}}+G\le t\right) \right| \\
         & \le 2 \sup_t\left|\p\left(G\le t\right)-\p\left( G\le t\right) \right| +  \frac{3}{2\sqrt{\pi}} \sqrt{\frac{1}{n}}\\
         & = \frac{3}{2\sqrt{\pi}}\frac{1}{\sqrt{n}}
    \end{align*}
    
    Again by  \cref{lemma:Normal-variation}, we see
    \begin{align*}
        \delta 
        & \le 2 \sup_t\left|\p\left(\frac{\xi+ \sum X_iY_i }{\sqrt{n}}\le t\right)-\Phi(t) \right| 
        +  \frac{3}{2\sqrt{\pi}}\frac{1}{\sqrt{n}} \\
        &\le 2 (\delta_{\xi} + \eta) + \frac{3}{2\sqrt{\pi}}\frac{1}{\sqrt{n}}\\
        &< 2 \delta_{\xi}+ \frac{3}{\sqrt{n}} 
    \end{align*}
\end{proof}

\bigskip
Now we are ready to prove the rate of convergence in  \cref{thm:rate of convergence}.\\
\begin{proof}
Let $\{Z_i\}$ be a sequence of independent standard normal random variables which is independent from $\{X_i , Y_i\}$.
By conditioning, we can rewrite $\delta_{\xi}$ of  \cref{lemma:normal_difference}
\begin{align*}
 \delta_{\xi} & = \sup_t \left| \p\left(\frac{\xi+ \sum X_iY_i }{\sqrt{n}}\le t\right)-\p\left(\frac{\xi }{\sqrt{n}}+G\le t\right) \right| \\
& = \sup_t \left| \p\left(\frac{\xi+ \sum X_iY_i }{\sqrt{n}}\le t\right)-\p\left(\frac{\xi+\sum Z_iY_i }{\sqrt{n}}\le t\right)  + \Delta_t \right| \\
& =\sup_t \left| \E \left[\sum_{m=1}^n \Phi \left(T_m \right)-\Phi \left(T_{m-1} \right)  \right] + \Delta_t \right|
\end{align*}
where 
\begin{align*}
     \Delta_t &= \p\left(\frac{\xi+\sum Z_iY_i }{\sqrt{n}}\le t\right) -\p\left(\frac{\xi }{\sqrt{n}}+G\le t\right)\\
    T_m &=t\sqrt{n}-\sum_{i=1}^m X_iY_i -\sum_{i=m+1}^n Z_iY_i 
\end{align*}
Therefore with  \cref{lemma:part_bound_1} controlling the part of telescopic sum and  \cref{lemma:part_bound_2} controlling $\sup_t|\Delta_t|$ (which we will prove later in \cref{subsection: lemma}), we see,
\[
\sup_t\left| \p\left(\frac{\xi+ \sum X_iY_i }{\sqrt{n}}\le t\right)-\p\left(\frac{\xi }{\sqrt{n}}+G\le t\right)  \right| \le O(\varepsilon_n \vee \frac{1}{\sqrt{n}} )
\]
Then by  \cref{lemma:normal_difference}, we conclude the desired result
\[
\sup_t\left|\p\left(\frac{\sum X_iY_i}{\sqrt{n}}\le t\right)-\Phi(t)\right| \le O(\varepsilon_n \vee \frac{1}{\sqrt{n}} )
\]
\end{proof}

\begin{remark}\label{Nonequi}
If we let $Y_k=1$ for all $k$, then we recover the rate of convergence $O(\frac{1}{\sqrt{n}})$ for a martingale difference sequence $\{X_k\}$. This is not contradicting the Martingale difference CLT which has a rate $O(\frac{\log n}{\sqrt{n}})$, see \cite{bolthausen1982}.  Martingale CLT is derived under a slightly  weaker condition on variance, which only requires $\frac{1}{n}\sum_k \E [X_k^2|\cF_{k-1}] \to 1 $ instead of our condition that $\E [X_k^2|\cF_{k-1}] $  to be constant $1$ for all $k$.

\end{remark}


\subsection{Proof of  \cref{lemma:part_bound_1} and \cref{lemma:part_bound_2}} \label{subsection: lemma}
\begin{lemma}\label{lemma:part_bound_1}
If $\E X_k^3<A<\infty, \E Y_k^3 <A< \infty , \forall k$ then there is a constant $c$
\begin{align}\label{eqn:Lindeberg_sum_rate}
    \sup_t \left| \E \left[\sum_{m=1}^n \Phi \left(T_m \right)-\Phi \left(T_{m-1} \right)  \right] \right| 
    \le \frac{c}{\sqrt{n}}
\end{align}
\end{lemma}
\begin{proof}
Let $U_k=T_k-X_kY_k=T_{k-1}-Z_kY_k$, then 
		\[
		\Phi(T_k) - \Phi(U_k)=\Phi'(U_k) \frac{1}{\sqrt{n}} X_kY_k +\frac{1}{2}\Phi''(U_k)\frac{1}{n} X_k^2 Y_k^2 +O(n^{-\frac{3}{2}} |X_k^3Y_k^3| \sup_x \Phi'''(x))
		\] 
		\[
		\Phi(T_{k-1}) - \Phi(U_k)=\Phi'(U_k) \frac{1}{\sqrt{n}} Z_k Y_k +\frac{1}{2}\Phi''(U_k)\frac{1}{n} Z_k^2 Y_k^2 +O(n^{-\frac{3}{2}} |Z_k^3 Y_k^3| \sup_x \Phi'''(x))
		\] 
Similar arguments from the CLT proof shows the first two terms match. Therefore 
\begin{align*}
    \left| \E \left[\sum_{m=1}^n \Phi \left(T_m \right)-\Phi \left(T_{m-1} \right)  \right] \right|  
    & \le \E \left[\sum_{m=1}^n O(n^{-\frac{3}{2}} (|X_k^3|+|Z_k^3|)|Y_k^3| \sup_x \Phi'''(x)) \right] \\
    &\le \frac{c}{\sqrt{n}}
\end{align*}
Note $\Phi'''(x)=\frac{x^2-1}{\sqrt{2\pi}} e^{-x^2/2}$ and $|\sup_x \Phi'''(x) |<\frac{2}{5}$.

\end{proof}

\begin{lemma}\label{lemma:part_bound_2}
If condition \cref{eqn:assumption_rate}
\begin{align*}
    \E \left[ 1 \wedge \left|\sqrt{\frac{\sum Y_k^2}{n}}-1 \right|  \right] \le O(\varepsilon_n)
\end{align*}
 is satisfied. Then 
\[
\sup_t |\Delta_t| \le O(\varepsilon_n \vee \frac{1}{\sqrt{n}} )
\]
\end{lemma}

\begin{proof}
With similar argument in  \cref{lemma:normal_difference}, we can removing the same variation term, normal random variable $\frac{\xi }{\sqrt{n}}$ in $\Delta_t$. So for some constant $c_0$,
\[
\sup_t |\Delta_t|\le 2\sup_t \left|\p\left(\frac{\sum Z_iY_i }{\sqrt{n}}\le t\right) -\p\left( G\le t\right) \right| +\frac{c_0}{\sqrt{n}}
\]

\begin{align*}
    \sup_t \left| \p\left(\frac{\sum Z_iY_i }{\sqrt{n}}\le t\right)-\p(G\le t) \right|
    &= \sup_t \E \left| \p (G \le t\sqrt{\frac{n}{\sum Y_i^2}} ) -\p(G\le t) \right| \\
    &= \sup_t \E \left| \int_{t}^{t\sqrt{\frac{n}{\sum Y_i^2}} } \frac{1}{\sqrt{2\pi}} e^{-\frac{x^2}{2}} dx \right|  \\
  &: = \sup_t \E h(t) \\
  & \le \E \sup_t h(t) 
\end{align*}
where we denote $S_n=\sqrt{\frac{\sum Y_i^2}{n}}$,  $h(t)=\left| \int_{t}^{t /S_n }  \frac{1}{\sqrt{2\pi}} e^{-\frac{x^2}{2}} dx \right|$. 

Notice $0<h(t)<1$ and $h(t)<|t-t/S_n|\frac{1}{\sqrt{2\pi}} e^{-\frac{\min({t}^2, {t}^2/ S_n^2)}{2}} $. So
\begin{align*}
    \sup_t |h(t)| & \le 1 \wedge \left[ |t/S_n - t| \frac{1}{\sqrt{2\pi}} e^{-\frac{\min({t}^2, {t}^2/S_n^2)}{2}} \right] 
\end{align*}

Notice the fact  $\sup_x \frac{1}{\sqrt{2\pi}}|x e^{-\frac{x^2}{2}}|< \frac{1}{2}$. When $S_n>1$, $\min({t}^2, {t}^2/S_n^2) = t^2/S_n^2$, we conclude 
$$  \sup_t |h(t)| \le 1 \wedge  \frac{1}{2}| 1-S_n| \le 1 \wedge  | 1-S_n| $$
When $\frac{1}{2}< S_n <1$, we have $4S_n^2>1$. Then $\min({t}^2, {t}^2/S_n^2) \ge \frac{t^2}{4S_n^2}$. We see 
\begin{align*}
    \sup_t |h(t)| & \le 1 \wedge \left[ |2 - 2S_n| \frac{t}{2S_n} \frac{1}{\sqrt{2\pi}} e^{-\frac{\frac{t^2}{4S_n^2}}{2}} \right]  \\
    & \le 1 \wedge \frac{1}{2}|2 - 2S_n| \\
    & = 1 \wedge  | 1-S_n|
\end{align*}
When $ S_n<\frac{1}{2}$, we use the bound $\sup_t |h(t)|<1$. And 
\begin{align*}
    \p(S_n <\frac{1}{2}) \le 2 \E[1\wedge |1-S_n|, S_n <\frac{1}{2}] 
\end{align*}

 Combining condition \cref{eqn:assumption_rate}, we conclude 
\begin{align*}
   \E [\sup_t h(t)] 
   & \le \E \left[ 1 \wedge | S_n - 1|, S_n >\frac{1}{2} \right] + \E \left[ 1 , S_n <\frac{1}{2} \right] \\
   & \le   O(\varepsilon_n) + \p\left( S_n <\frac{1}{2} \right) \\
   & \le O(\varepsilon_n) + 2 \E[1\wedge |1-S_n|, S_n <\frac{1}{2}] \\
   & \le O(\varepsilon_n)
\end{align*}

Then we conclude 
\begin{align*}
    \sup_t |\Delta_t|\le O( \varepsilon_n \vee \frac{1}{\sqrt{n}} )
\end{align*}

\end{proof}


\subsection{Discussion on the assumptions }

A natural question is whether the condition \cref{eqn:assumption_rate} is necessary for  \cref{thm:rate of convergence}.  We will first show the condition \cref{eqn:assumption_rate} for  \cref{lemma:part_bound_2} is sharp by obtaining a lower bound for $\sup_t \E h(t)$.  This implies the technique we used in proving  \cref{thm:rate of convergence} is delicate enough to squeeze out any unnecessary relaxation. 
\begin{proposition}\label{prop:assumption_sharp_lowerbound}
        \begin{align}
        \sup_t \left| \p\left(\frac{\sum Z_iY_i }{\sqrt{n}}\le t\right)-\p(G\le t) \right|: = \sup_t \E h(t) \ge O(\E \left[ 1 \wedge \left|\sqrt{\frac{\sum Y_k^2}{n}}-1 \right| \right] )
        \end{align}
\end{proposition}
\begin{proof}
Take $t=1$ we find 
\begin{align*}
    \E h(1) & = \E \left| \int_{1}^{1 /S_n }  \frac{1}{\sqrt{2\pi}} e^{-\frac{x^2}{2}} dx \right| \\
    & \ge  c \; \E \left[  \int_{1}^{1 /S_n } dx, \frac{1}{S_n} \le 2 \right] +  \; \E \left[\int_1^2  \frac{1}{\sqrt{2\pi}} e^{-\frac{x^2}{2}} dx , \frac{1}{S_n} > 2 \right] \\
    & \ge  c \; \E \left[|1- \frac{1}{S_n}|, S_n \ge \frac{1}{2} \right] + c \; \E \left[1,  S_n < \frac{1}{2}  \right] 
\end{align*}
where $c= \frac{1}{\sqrt{2\pi}} e^{-\frac{2^2}{2}}\ge \frac{1}{20} $.

We will further separate $ S_n \ge \frac{1}{2} $ into three events. 
\begin{align*}
  \frac{1}{2} \le S_n\le 1 
    & :\quad
    \frac{1}{S_n} -1 \ge 1- S_n \ge 0\\
  1< S_n < 2 
    & :\quad
    1-\frac{1}{S_n}  \ge \frac{1}{2} ( S_n -1)\ge 0 \\
    S_n \ge 2 
    & :\quad
    1-\frac{1}{S_n}  \ge \frac{1}{2} \\
\end{align*}
So overall on the event $ S_n \ge \frac{1}{2} $, we have 
\[
|\frac{1}{S_n} -1 | \ge \frac{1}{2} [ 1 \wedge |S_n -1| ]
\]
Combining all together, 
\begin{align*}
    \sup_t \E h(t) \ge \E h(1) 
    & \ge \frac{1}{40} \E \left[ 1 \wedge |S_n -1|, S_n \ge \frac{1}{2} \right] + \frac{1}{20}  \E \left[1,  S_n < \frac{1}{2}  \right] \\
   & \ge  \frac{1}{40}  \E \left[ 1 \wedge |S_n -1| \right]
\end{align*}
\end{proof}

\bigskip
Next we will use i.i.d. $X_i, Y_i$ as an example to show the rate of convergence obtained from  \cref{thm:rate of convergence} is the same as Berry-Esseen in classical CLT. This implies the condition \cref{eqn:assumption_rate} is sharp for this specific example. Note this does not imply condition \cref{eqn:assumption_rate} is sharp in general. However, any nontrivial improvement would require more restrictive assumptions.
\begin{proposition}[] \label{prop:assumption is sharp}
In \cref{thm:rate of convergence}, condition \cref{eqn:assumption_rate} is sharp if $X_i, Y_i$ are i.i.d. sequences with mean zero and variance one.
\end{proposition}
\begin{proof}
    Let   $\{X_i\}, \{ Y_i\}$ be i.i.d. random variables with mean zero, variance one (e.g. standard normal). Then by the classical CLT and Berry-Esseen we know 
    \[
	\p(\frac{1}{\sqrt{n}}\sum_{k=1}^n X_k Y_k<t) =\p(G<t)  +O(\frac{1}{\sqrt{n}} ) \quad \forall t\in \R 
	\]

Let us derive the same result from \cref{thm:rate of convergence}. For i.i.d. $Y_i$ mean zero and variance one, since $\sqrt{n}(\sum Y_i^2/n-1) \to \cN(0,1)$ and $(\sqrt{\sum Y_i^2/n}+1) \to 2$ in probability and we can apply Slutsky's theorem, $$\sqrt{n}\left(\sqrt{\frac{\sum Y_i^2}{n}}-1 \right) = \sqrt{n}\frac{(\sum Y_i^2/n-1)}{(\sqrt{\sum Y_i^2/n}+1)} \to \cN(0,\frac{1}{4})$$ 
 Therefore condition \cref{eqn:assumption_rate} is satisfied with
\begin{align*}
    \E \left( 1 \wedge \left| \sqrt{\frac{\sum Y_i^2}{n}}-1 \right|  \right) 
    & = O(\frac{1}{\sqrt{n}}) 
\end{align*}

Then \cref{thm:rate of convergence} gives the same conclusion as Berry-Esseen.
\end{proof}

\bigskip
\bigskip

A more intuitive control of the LLN of $Y_k^2$ would be controlling the tail probability directly, which will not be sharp. 
\begin{proposition}\label{prop:assumption_rate_tail_prob}
In \cref{thm:rate of convergence}, condition \cref{eqn:assumption_rate} can be replaced by
        \begin{align}\label{eqn:assumption_rate_p}
        \p \left( \left|\sqrt{\frac{\sum Y_k^2}{n}}-1 \right| > O(\varepsilon_n) \right) \le O(\varepsilon_n)
        \end{align}
where $\varepsilon_n \to 0$.  This condition is stronger than \cref{eqn:assumption_rate}.  In other words, it is sufficient for \cref{thm:rate of convergence} but not necessary.
\end{proposition}
\begin{proof}
    Let $S_n = \sqrt{\frac{\sum Y_k^2}{n}}$. Assume \cref{eqn:assumption_rate_p} holds.
\begin{align*}
     \E \left[ 1 \wedge |S_n-1|  \right]
    & = \E \left[ 1 \wedge |S_n-1|, |S_n-1|>O(\varepsilon_n)  \right] 
       + \E \left[ 1 \wedge |S_n-1|, |S_n-1|\le O(\varepsilon_n)  \right] \\
    & \le  \p(|S_n-1|>O(\varepsilon_n) ) + O(\varepsilon_n) \\
    & \le O(\varepsilon_n)
\end{align*}

{
To show condition \cref{eqn:assumption_rate_p} is stronger than condition \cref{eqn:assumption_rate}, we look at the example of i.i.d. $\{Y_i\}$ sequence.

For i.i.d. $Y_i$ mean zero and variance one,  $$\sqrt{n}\left(\sqrt{\frac{\sum Y_i^2}{n}}-1 \right) = \sqrt{n}\frac{(\sum Y_i^2/n-1)}{(\sqrt{\sum Y_i^2/n}+1)} \to \cN(0,\frac{1}{4})$$ 
Since $\sqrt{n}(\sum Y_i^2/n-1) \to \cN(0,1)$ and $(\sqrt{\sum Y_i^2/n}+1) \to 2$ in probability and we can apply Slutsky's theorem. Therefore condition \cref{eqn:assumption_rate} is satisfied
\begin{align*}
    \E \left( 1 \wedge \left| \sqrt{\frac{\sum Y_i^2}{n}}-1 \right|  \right) 
    & = O(\frac{1}{\sqrt{n}}) 
\end{align*}
However, condition \cref{eqn:assumption_rate_p} is not  satisfied since 
\begin{align*}
    \p \left( \left|\sqrt{\frac{\sum Y_k^2}{n}}-1 \right| > O(\frac{1}{\sqrt{n}}) \right) 
    = \; &  \p \left(\sqrt{n} \left|\sqrt{\frac{\sum Y_k^2}{n}}-1 \right| > O(1) \right) \\
    \approx \; & \p \left(|\cN(0,\frac{1}{4})| > O(1) \right)\\
    = \; & O(1)
\end{align*}
}
\end{proof}


\section{Random projection}\label{sec:random projection}

Suppose we have  $X, Z$  two independent random vectors. In this section, we will investigate how much the independence structure is preserved in the projected space. Let $S$ be a random projection, the resulting projected random vectors $SX, SZ$ will be dependent. We will see the distribution of inner product is preserved under certain conditions.

Given two independent random vectors in $\R^n$: 
$$X=\mat[c]{X_1\\ \vdots\\ X_n}, Z=\mat[c]{Z_1\\ \vdots\\ Z_n}$$
with  i.i.d. entries and all $X_i, Z_i$ independent with each other. Let  $ \E X_i = \E Z_i = 0, \quad \E X_i^2 = \E Z_i^2 = 1, \quad \E |X_i|^3 \vee \E |Z_i|^3 < A <\infty $. Then it is clear the following CLT holds:
\[
\frac{1}{\sqrt{n}}X^T Z = \frac{1}{\sqrt{n}}\sum_{i=1}^n X_i Z_i \xrightarrow{\; d \;} \cN(0,1)
\]
And the classical  Berry-Esseen theorem \cite{berry1941accuracy, esseen1969remainder, durrett2019probability} tells us 
\begin{align}\label{eqn:Berry-Esseen classical}
    \sup_t \left| \p\left(\frac{1}{\sqrt{n}}X^T Z <t  \right) - \p\left(\cN(0,1) <t  \right) \right| \le O \left(\frac{1}{\sqrt{n}} \right)
\end{align}

Consider a  random  matrix $S: \R^n \to \R^m$ whose entries has mean 0 and variance 1. Then the natural question is whether CLT holds for product of the randomly projected vectors $SX$ and $SZ$. Namely 
\[
\frac{1}{\sqrt{n}\; a_{n,m}} X^T S^T S Z \xrightarrow{\; ? \;} \cN(0,1)
\]
where $a_{n,m}$ is a scaling parameter depend on both $m$ and $n$. Moreover, we will need to derive the rate of convergence in the spirit of Berry-Esseen theorem, namely find
\[
\sup_t \left| \p\left(\frac{1}{\sqrt{n}\; a_{n,m}} X^T S^T S Z <t  \right) - \p\left(\cN(0,1) <t  \right) \right| \le \; ? 
\]

If we try to use existing CLT that deals with dependent random variables, for example martingale CLT, it will not be applicable. The major difficulty is that there is no natural filtration since the terms in the sum will be very dependent so the conditional variance in martingale CLT is not computable. It turns out our product-CLT is the right tool to use. We decouple the dependence into the sequence of independent random variables $X$ and another sequence $S^TSZ$ with complicated dependence.


Now what are the necessary conditions required to apply our product-CLT?	Since $\{X_i\}$ is a sequence with independent random variables, it satisfies all conditions in \cref{thm:product_CLT} and \cref{thm:rate of convergence}. So we need to show the assumptions on the second dependent sequence 
	$$Y = \mat[c]{Y_1\\ \vdots \\ Y_n }:=\frac{1}{a_{n,m}} S^{T} SZ$$
	is also satisfied. Denote $i$-th column of $S$ as $S_i$, then $Y_i = \frac{1}{a_{n,m}} S_i^T SZ$. Moreover, $\{Y_i\}$ are identically distributed even though they are dependent random variables.
The Lindeberg swap idea in \cref{thm:product_CLT} requires the variables $ Y_i$ have finite third moments and a weak law of large number of $Y_i^2$. We shall prove the weak law of large number in the following \cref{lemma:random projection weak LLN Y^2}. In the proof we will follow \cref{prop:second moment condition} using Chebyshev's inequality to show the weak law of large number statement. To find the rate of convergence, we will need to compute the exact order of \cref{eqn:assumption_rate}.
	
{
    
	
}

\subsection{Random matrix preserves inner product}
	
	\begin{lemma}\label{lemma:random projection weak LLN Y^2}
		Given $m, n \to \infty$, \; $a_{n,m}=\sqrt{m^2+mn}$, and $\E S_{i,j}^4 \vee \E S_{i,j}^8 \vee \E z_i^4 < C <\infty$. If we let 
		\[
		y_i = \frac{1}{a_{n,m}} S_i^T SZ
		\]
		then we have 
		\[
		\E y_i^2 \to 1, \forall i
		\]
		and 
		\[
		\frac{1}{n} \sum_{i=1}^n y_i^2\to 1
		\]
	\end{lemma}
	\begin{proof}
		First, we note $y_i$ are identically distributed. By  \cref{prop:second moment condition}, it suffices to prove $\E y_i^2 \to 1$, $\E y_i^4<C <\infty$, and $\E y_i^2 y_j^2 \to 1$. 
		
		For the second moment,
		\begin{align*}
		\E [y_1^2]&{}=\frac{1}{a_{n,m}^2}\E [(S_1^T S Z)^2] \\
		&{}=\frac{1}{a_{n,m}^2 }\sum_{ 1\le i,j\le m, 1\le p, q \le n }\E  S_{i,1} S_{i,p} z_{p}  S_{j,1} S_{j,q} z_{q} \\
		&{}=\frac{1}{a_{n,m}^2 }\sum_{ 1\le i,j\le m, 1\le p, q \le n }\E  S_{i,1} S_{i,p}  S_{j,1} S_{j,q} \E z_{p} z_{q}
		\end{align*}	
		Notice the random matrix $S$ and random vector $Z$ are centered,  $\E S=0$, $\E Z=0$. The surviving terms have to be even powers, which are $\{p=q\ne 1,  i=j\}$, $\{p=q=1, i,j \}$. Therefore
		\begin{align}\label{eqn:Y_squared}
		\E [y_1^2] &{} 
		= \frac{1}{a_{n,m}^2} \left[ \sum_{ 1\le i\le m, 2\le p \le n }\E  S_{i,1}^2 S_{i,p}^2 \E z_{p}^2 + \sum_{ 1\le i, j\le m }\E  S_{i,1}^2 S_{j,1}^2 \E z_{1}^2 \right] \nonumber\\
		& = \frac{1}{(m^2+mn)} [m(n-1) + (m\E S_{1,1}^4 +m^2-m)] \nonumber\\
		& =  1 + \frac{\E S_{1,1}^4-2}{m+n} \nonumber\\
		&{} = 1 +O\left(\frac{1}{m+n}\right)
		\to  1 
		\end{align}

		Now we will show $\E y_i^2 y_j^2 \to 1$ for all $i\ne j$. 
		\begin{align*}
		\E [y_1^2 y_2^2]&{}=\frac{1}{a_{n,m}^4}\E [(S_1^T SZ)^2 (S_2^T SZ)^2] \\
		&{}=\frac{1}{a_{n,m}^4}\sum_{}\E  (S_{i_1,1} S_{i_1,p_1} S_{j_1,1} S_{j_1,q_1} z_{p_1} z_{q_1})\; (S_{i_2,2} S_{i_2,p_2} S_{j_2,2} S_{j_2,q_2} z_{p_2} z_{q_2}) 	
		\end{align*}
		First, there are eight indices in the summation. And $1\le i_1, i_2, j_1, j_2 \le m$,  $ 1\le p_1, q_1, p_2, q_2 \le n$.  Since 
		$\E S_{i,j}= 0, \E z_i = 0$, the surviving terms in the summation must have higher powers for $S_{i,j}$ and $z_i$. We will count the total number of possible such terms. 
		
		Surviving terms will satisfy the following condition
		$$z_{p_1} z_{q_1} z_{p_2} z_{q_2}= z_p^2 z_q^2, \quad 1\le p, q \le n$$ 
		
		We will analyze and count in two different cases:
		$$
		 \{p, q\} \cap \{1,2\} \ne \emptyset, \quad  \{p, q\} \cap \{1,2\} = \emptyset 
		$$
There are still many sub-cases, we need to treat differently. 		
\begin{itemize}
 \item Case 1: $\{p, q\} \cap \{1,2\} \ne \emptyset$
    \begin{itemize}
    \item Case 1-1: $\{p, q \}\subseteq \{ 1, 2\}$.
        \begin{itemize}
            \item Case 1-1-1: $p=q=1$. Then each term  is $\E S_{i_1,1}^2 S_{j_1,1}^2 S_{i_2,2} S_{i_2, 1} S_{j_2,2} S_{j_2, 1} z_1^4$. Then $i_2=j_2$ in order to have squares.  So the total is 
            \[
            m^3 \E z_1^4 + O(m^2)
            \]
            \item Case 1-1-2: $p=q=2$. Same as the computation in case 1-1-1, we have total
            \[
            m^3 \E z_1^4 + O(m^2)
            \]
            \item Case 1-1-3: $p=1,q=2$. This will give us  $\binom{4}{2}=6$ separate cases.
        \begin{center}
         \begin{tabular}{|c|c|c|c|} \hline
            $p_1$ & $q_1$ & $p_2$ & $q_2$ \\ \hline
            1 & 1 & 2 & 2 \\
            1 & 2 & 1 & 2 \\
            1 & 2 & 2 & 1 \\
            2 & 1 & 1 & 2 \\
            2 & 1 & 2 & 1 \\
            2 & 2 & 1 & 1 \\\hline
         \end{tabular} 
        \end{center}
            Only the first case (1, 1, 2, 2) produces terms $\E S_{i_1,1}^2 S_{j_1,1}^2 S_{i_2,2}^2 S_{j_2,2}^2  z_1^2 z_2^2$ . In total it is $m^4 +O(m^3)$.         
         All other five cases admit similar analysis with same number of terms, we only show the second case (1, 2, 1, 2), which is  
         
         $\E S_{i_1,1}^2 S_{j_1,1} S_{j_1,2} S_{i_2,2}S_{i_2,1} S_{j_2,2}^2  z_1^2 z_2^2$. Then $j_1 = i_2$ must hold for the surviving terms, which in total is $m^3 +O(m^2)$. Combining all together, we have in total
         \[
         m^4 +O(m^3)
         \]

        \end{itemize}
    
    \item Case 1-2: $p=1, q \notin  \{ 1, 2\}$. Same as 1-1 there are  $\binom{4}{2}=6$ separate cases.
        \begin{center}
         \begin{tabular}{|c|c|c|c|} \hline
            $p_1$ & $q_1$ & $p_2$ & $q_2$ \\ \hline
            1 & 1 & q & q \\
            1 & q & 1 & q \\
            1 & q & q & 1 \\
            q & 1 & 1 & q \\
            q & 1 & q & 1 \\
            q & q & 1 & 1 \\\hline
         \end{tabular} 
        \end{center}

           Only the first case (1, 1, q, q) produces terms $\E S_{i_1,1}^2 S_{j_1,1}^2 S_{i_2,2}S_{i_2,q} S_{j_2,2}S_{j_2,q}  z_1^2 z_q^2$. In this case $i_2 =j_2$ must hold. In total, there are $m^3(n-2) +O(m^2n)$ terms.          
           
         All other five cases have similar analysis with same number of terms, we only show the the second case (1, q, 1, q), $\E S_{i_1,1}^2 S_{j_1,1}S_{j_1,q} S_{i_2,2}S_{i_2,1} S_{j_2,2}S_{j_2,q}  z_1^2 z_q^2$. In this case $j_1= i_2 =j_2$ must hold. In total, there are $m^2(n-2) +O(mn)$ terms.  Combining all together, we have in total
         \[
         m^3n +O(m^3+m^2n)
         \]
    \item Case 1-3: $p=2, q \notin  \{ 1, 2\}$. Again there are    $\binom{4}{2}=6$ separate cases.
         \begin{center}
         \begin{tabular}{|c|c|c|c|} \hline
            $p_1$ & $q_1$ & $p_2$ & $q_2$ \\ \hline
            2 & 2 & q & q \\
            2 & q & 2 & q \\
            2 & q & q & 2 \\
            q & 2 & 2 & q \\
            q & 2 & q & 2 \\
            q & q & 2 & 2 \\\hline
         \end{tabular} 
        \end{center}
           Only the last case (q, q, 2, 2) produces terms $\E S_{i_1,1} S_{i_1,q} S_{j_1,1}S_{j_1,q} S_{i_2,2}^2 S_{j_2,2}^2  z_q^2 z_2^2$. Then $i_1=j_1$ must hold. In total it is $m^3(n-2) +O(m^2n)$.    
           
         All other five cases have similar analysis, we only show the first (2, 2, q, q)here. \\
         $\E S_{i_1,1}S_{i_1,2} S_{j_1,1} S_{j_1,2}$ $  S_{i_2,2} S_{i_2,q} S_{j_2,2}S_{j_2,q}  z_q^2 z_2^2$. Then $i_1=j_1,i_2 = j_2$ must hold for the surviving terms, which in total is $m^2(n-2) +O(mn)$. Combining all together, we have in total
         \[
         m^3n +O(m^3+m^2n)
         \]
    \end{itemize} 
 \item Case 2: $\{p, q\} \cap \{1,2\} = \emptyset$. \\
 To have squares for variables from matrix $S$, we must have squares produced for $S_{i_1,1} S_{j_1,1} S_{i_2,2} S_{j_2,2}$ and $S_{i_1,p_1} S_{j_1,q_1} S_{i_2,p_2} S_{j_2,q_2}$ separately. Therefore $i_1=j_1, i_2=j_2$. Denote $i_1:= i, j_1:= j$. Then we can further split into two cases, $i=j$ and $i\ne j$.
   \begin{itemize}
       \item Case 2-1: $\{p, q\} \cap \{1,2\} = \emptyset$ and $i =j $. Then each term involving $S$ is $\E S_{i,1}^2 S_{i,2}^2 S_{i,p_1} S_{i, p_2} S_{i,q_1} S_{i, q_2}$. This  will produce 3 possible matches for $\{ p_1, p_2, q_1, q_2\}= \{p, q\}$.  which counting all indices will yields total $3m(n-2)^2$ terms. Some of those terms will have $p=q$, which will produce $m(n-2) \E S_{1,1}^4 \E z_1^4$ which is of a smaller order. So total will be 
       \[
       3m n^2 +O(mn)
       \]
       \item Case 2-2: $\{p, q\} \cap \{1,2\} = \emptyset$ and $i \ne j $. In this case $\{p_1=q_1, p_2=q_2$ must be true. That in total will produce $(m^2-m)[(n-2)^2-(n-2)]$ terms which we excluded the cases when $p=q$. Then the cases of $p=q$  in total are  $(m^2-m)(n-2)$ of $\E z_1^4$. In total
       \begin{align*}
        & (m^2-m)[(n-2)^2-(n-2)] + (m^2-m)(n-2) \E z_1^4 \\
        = & m^2n^2 - mn^2 +(\E z_1^4 -5) m^2n + O(mn + m^2)    
       \end{align*}
   \end{itemize}
\end{itemize}

Adding all the cases together we obtain 
\begin{align}
     \E [y_1^2 y_2^2] = & \frac{1}{(m^2+mn)^2} [m^4+ 2m^3n + m^2n^2 +O(m^3 +m^2n +mn^2) ] \\
     = & 1+ \frac{O(m^3 +m^2n +mn^2) }{(m^2+mn)^2} \\
     = & 1+ O\left(\frac{1}{m}+ \frac{1}{m+n} \right) \to 1
\end{align}
Therefore,
\[
\frac{1}{n^2} \sum_{i\ne j} \E [(y_i^2-1) (y_j^2-1)] = \frac{n^2-n}{n^2}\E [y_1^2 y_2^2 - y_1^2 - y_2^2 +1] \to 0 
\]
		
		For the fourth moment, we will show $\E y_i^4 \le  C<\infty$.
		\begin{align*}
		\E [y_1^4]&{}=\frac{1}{a_{n,m}^4}\E [(S_1^T SZ)^4] \\
		&{}=\frac{1}{a_{n,m}^4}\sum_{}\E  S_{i_1,1} S_{i_1,p_1} S_{j_1,1} S_{j_1,q_1} z_{p_1} z_{q_1} \\
		&\hspace{6em} S_{i_2,1} S_{i_2,p_2} S_{j_2,1} S_{j_2,q_2} z_{p_2} z_{q_2}
		\end{align*}
		Similarly, the surviving terms are $\{p_1=q_1, p_2=q_2, i_1=j_1, i_2=j_2\}$, $\{p_1=p_2, q_1=q_2, i_1=i_2, j_1=j_2\}$  and $\{p_1=q_2, q_1=p_2, i_1=j_2, j_1=i_2\}$ which in total will  give $3m^2n^2+  m^4 +  6m^3n +O(m^3+ m^2n+mn^2)$ where $m^4$ comes from counting terms of the form $\{ p_1=q_1= p_2=q_2=1 \}$, and $m^3n$ comes from 
		\[
		\{ p_1,q_1, p_2,q_2\}=\{1,q\}
		\]
		Therefore 
		\[
		\E y_i^4  = 3 \frac{n}{m+n}+ \frac{m^2}{(m+n)^2} + \frac{O(m^3 +m^2n +mn^2) }{(m^2+mn)^2} \le  4+ O\left(\frac{1}{m}+ \frac{1}{m+n} \right) 
		\]

		Lastly we shall apply Chebyshev's inequality.
		\begin{align*}
		\p\left(|\frac{1}{n} \sum y_i^2 -1| >\varepsilon \right) 
		&\le \frac{1}{n^2 \varepsilon^2} \left[ \sum \E (y_i^2-1)^2 +\sum_{i\neq j} \E (y_i^2-1)(y_j^2-1)  \right]  \to 0
		\end{align*}
	
	\end{proof}

\begin{proof}\textbf{(\cref{thm:rand_proj CLT} Random matrix inner product CLT)} \\
    
Combing \cref{lemma:random projection weak LLN Y^2} with the product-CLT  \cref{thm:product_CLT}, we conclude random projection preserves product-independence, namely given conditions in \textbf{\cref{thm:rand_proj CLT} } we have
\[
\frac{1}{\sqrt{m^2n+mn^2}} X^TS^TSZ \to \cN(0,1)  \qquad \text{as } m,n \to \infty
\]

\end{proof}	
\bigskip

Now we shall discuss the rate of convergence. Obtaining the exact rate is usually very hard since one has to compute the exact rate of convergence for the law of large number statement on $Y_k^2$ (namely condition \cref{eqn:assumption_rate}), 
which is not practically computable if no further information given. 
However it is possible to carry out an argument (for example using relaxations or \cref{prop:assumption_rate_tail_prob})  to obtain an upper bound. 

\bigskip

\begin{proof}\textbf{(\cref{thm:rand_proj rate of CLT} Random matrix invariance principle)} \\
    We will start with relaxation.  Since $\left| \sqrt{\frac{\sum Y_k^2}{n}}-1 \right| \le \left| \frac{\sum Y_k^2}{n}-1 \right|$
    \begin{align*}
    \E \left( 1 \wedge \left| \sqrt{\frac{\sum Y_k^2}{n}}-1 \right|  \right)  
    & \le  \E \left( 1 \wedge \left| \frac{\sum Y_k^2}{n}-1 \right|  \right)  \\
    & \le  \E \left[ \left| \frac{\sum Y_k^2}{n}-1 \right|  \right]  \\
    & \le \sqrt{\E \left[ \left( \frac{\sum Y_k^2}{n}-1 \right)^2  \right] }
    \end{align*}
The last step uses Jensen's inequality and $f(x)= \sqrt{x}$ is concave. Notice,
\begin{align*}
    \E \left[ \left( \frac{\sum Y_k^2}{n}-1 \right)^2  \right] 
    & = \frac{1}{n^2} \left[ \sum_{k=1}^n \E (Y_k^2 -1)^2  + \sum_{i\ne j}^n \E (Y_i^2 -1)(Y_j^2-1) \right] \\ 
    & =\frac{1}{n}\E (Y_1^2 -1)^2 +\frac{n^2-n}{n^2}  \E (Y_1^2 -1)(Y_2^2-1) 
\end{align*}

Therefore computing a bound for the rate of convergence boils down to compute the order of $\E (Y_1^2 -1)^2$ and $\E (Y_1^2 -1)(Y_2^2-1)$ explicitly which have already been computed in the proof of \cref{lemma:random projection weak LLN Y^2}.
\begin{align*}
		\E [Y_2^2] = \E [Y_1^2] 
		&{} = 1+  O\left(\frac{1}{m+n}\right) \\
		\E [Y_1^4] & \le 4 + O\left(\frac{1}{m}+ \frac{1}{m+n}\right) \le 4 + O\left(\frac{1}{m}\right) \\
		 \E Y_1^2 Y_2^2 
		 & = 1 +  O\left(\frac{1}{m}+ \frac{1}{m+n}\right) = 1 +  O\left(\frac{1}{m}\right) 
\end{align*}

This implies 
\[
\E (Y_1^2 -1)^2 = O\left(\frac{1}{m}\right) ,\quad \E (Y_1^2 -1)(Y_2^2-1) = O\left(\frac{1}{m} \right) 
\]
and 
So we conclude 
\begin{align*}
    \E \left( 1 \wedge \left| \sqrt{\frac{\sum Y_k^2}{n}}-1 \right|  \right)  
    & \le \sqrt{ \E \left[ \left( \frac{\sum Y_k^2}{n}-1 \right)^2  \right] }\\
    & = O\left( \frac{1}{\sqrt{m}} \right) 
\end{align*}
Applying \cref{thm:rate of convergence}, we conclude \cref{eqn:rand_proj rate to normal}. Then combining Berry-Esseen inequality \cref{eqn:Berry-Esseen classical} and triangle inequality, we obtain \cref{eqn:rand_proj rate to original vector}.
\end{proof}

 \bigskip


\subsection{Simulation}
We first give some simulations to show the random embedded or projected inner product converges to normal distribution.
\begin{figure}[H]
	\includegraphics[width=0.95\linewidth]{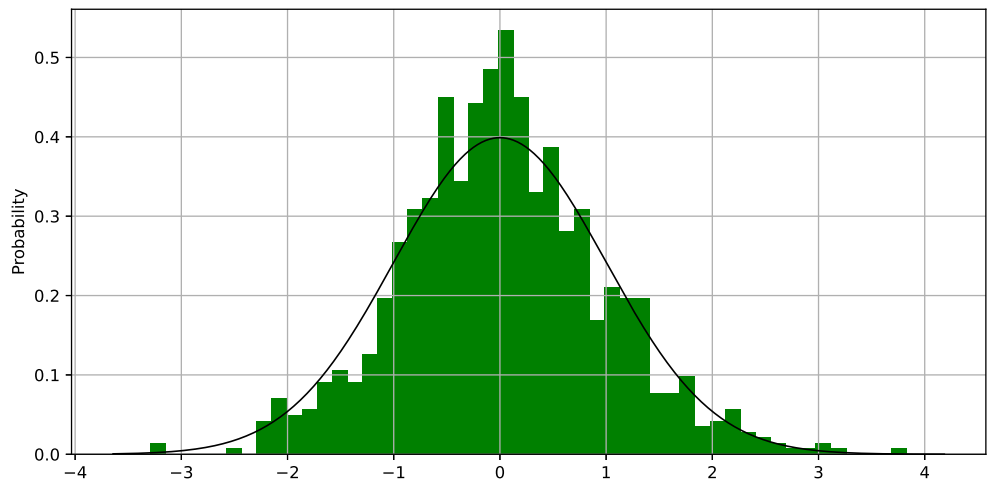}
	\caption{Random projected inner product (m=10, n=100)}
	\label{fig: inner_1 Rand Proj}
\end{figure}
\begin{figure}[H]
    \includegraphics[width=0.95\linewidth]{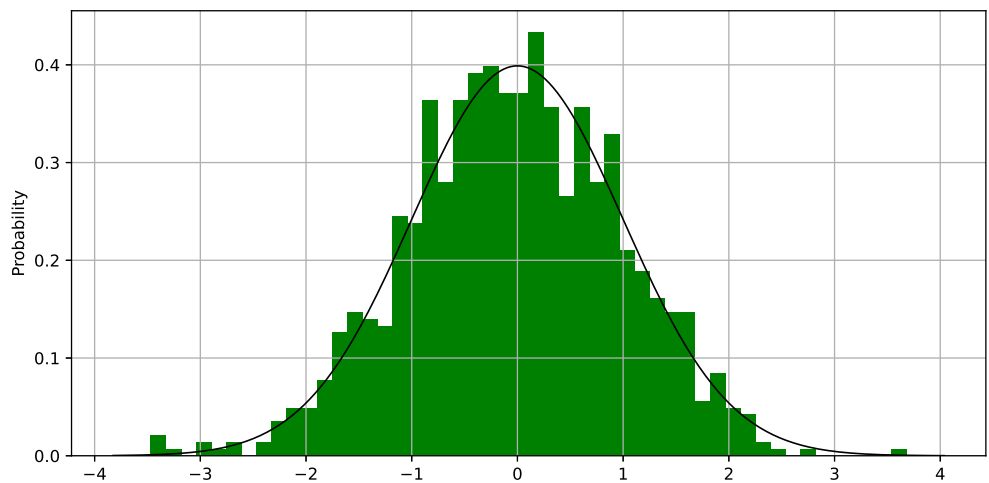}
    \caption{Random projected inner product (m=500, n=5000)}
    \label{fig: inner_2 Rand Proj}
\end{figure}

We have \cref{fig: inner_1 Rand Proj} and \cref{fig: inner_2 Rand Proj} plotted histograms of 1000 samples of the projected inner product $\frac{1}{\sqrt{m^2n+mn^2}} X^TS^TSZ$ with different dimension settings. The random variables we used for $X, S, Z$ are standard normal random variables. As dimension $m, n$ increases, the convergence improves.

Next we give simulations for random embedded inner product where $m> n$.
\begin{figure}[H]
	\includegraphics[width=0.95\linewidth]{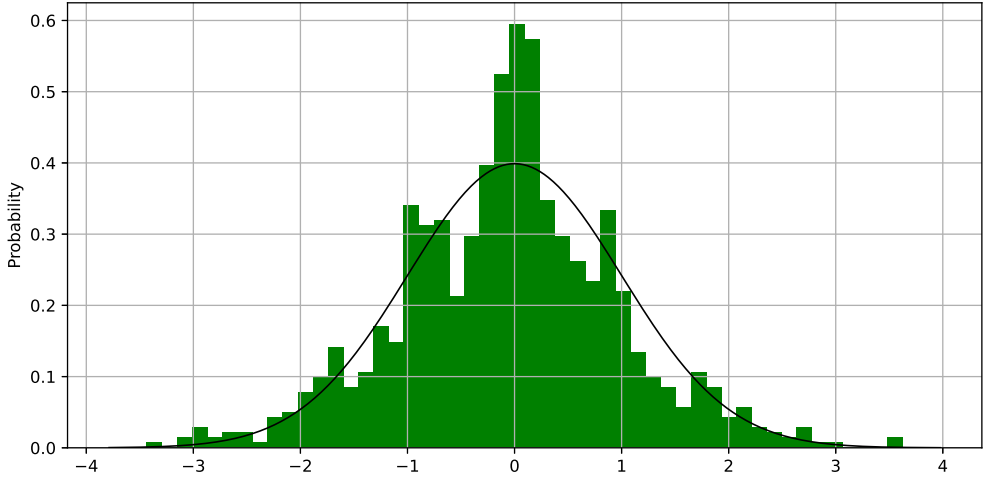}
	\caption{Random embedded inner product (m=500, n=50)}
	\label{fig: inner_1E Rand Proj}
\end{figure}
\begin{figure}[H]
    \includegraphics[width=0.95\linewidth]{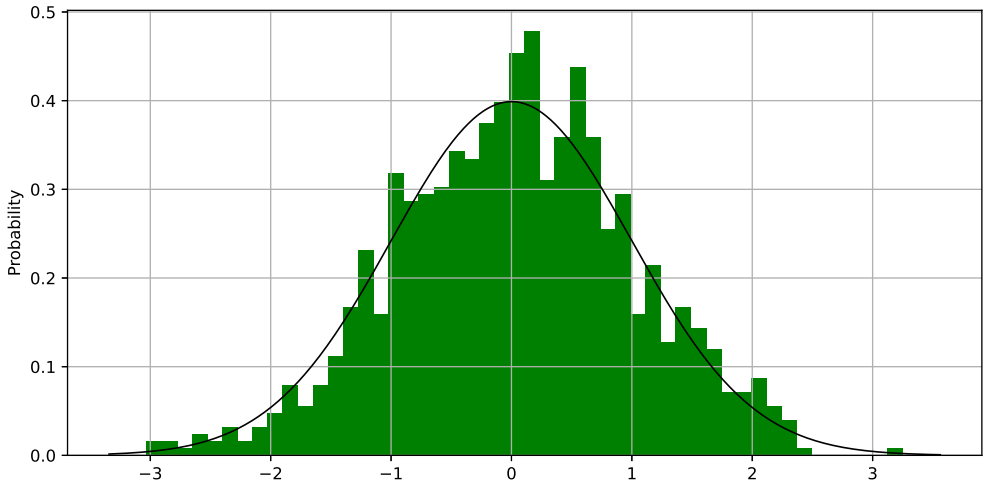}
    \caption{Random embedded inner product (m=5000, n=500)}
    \label{fig: inner_2E Rand Proj}
\end{figure}

Again \cref{fig: inner_1E Rand Proj} and \cref{fig: inner_2E Rand Proj} 
plotted histograms of 1000 samples of the embedded inner product $\frac{1}{\sqrt{m^2n+mn^2}} X^TS^TSZ$ . The random variables we used for $X, S, Z$  take discrete values $\{-2.5, 0, 2.5\}$ with probability $\{0.08, 0.84, 0.08\}$. The kurtosis of such random variable is $6.25$ which is much larger than standard normal random variable. Again as dimensions $m, n$ increase, the histogram converges to a standard normal shape.

\subsection{Discussion and open questions}
To have CLT result in \cref{thm:rand_proj CLT}, it is essential the dimension of the projected space $m$ diverges. Fixed $m$ will not lead to a CLT. 

For example, we let $m=1, n \to \infty$.  Then let $X \in \R^{n}$ be Gaussian vector, $Z\in \R^{n}$ be Rademacher vector and let $S: \R^n \to \R$ has  Rademacher entries. Suppose all random variables are independent, then 
$$\frac{1}{\sqrt{n}} X^TZ = \cN(0,1)$$
which holds exactly without error. On the other hand 
$$\frac{1}{\sqrt{1^2n+1n^2}} X^TS^TSZ \; \sim \; \cN(0,1) \times \cN'(0,1) +O(\frac{1}{\sqrt{n}})$$
that is the product of two independent standard Gaussian random variable. To see this is the case, note first $\frac{1}{\sqrt{n}} X^T S^T$ is exactly standard Gaussian $\cN(0,1)$. $\frac{1}{\sqrt{n+1}} S Z$ converges to another $\cN'(0,1)$ with error $O(\frac{1}{\sqrt{n}})$. The independence is due to the fact that Rademacher in $S$ can be absorbed into $X$ and $Z$ so that we may replace all entries of $S$ by constant $1$'s. Therefore the cdf of $\frac{1}{\sqrt{n}} X^TZ$ and $\frac{1}{\sqrt{1^2n+1n^2}} X^TS^TSZ$ differ by $O(1)=O(\frac{1}{\sqrt{m}})$.

The bound \cref{eqn:rand_proj rate to normal} in general can not be improved if there is no additional assumption.  $O(\frac{1}{\sqrt{n}})$ is necessary as it is in Berry-Esseen. $O(\frac{1}{\sqrt{m}})$ is also very likely to be necessary as the above example achieves the error rate when $m=1$.  For general $m$ we do not pursue a precise proof here but we give some heuristics.
Let $X \in \R^{n}$ be standard Gaussian vector, $Z\in \R^{n}$ be standard Rademacher vector and let $S: \R^n \to \R^m$ has  Rademacher entries as well. Suppose all random variables are independent. Denote 
$Y = \frac{1}{\sqrt{m^2+mn}} S^TSZ$. Notice $\frac{1}{\sqrt{n}} X^TZ$ is a standard Gaussian variable.
By the proof in  \cref{lemma:part_bound_2} and  \cref{prop:assumption_sharp_lowerbound}, we have the lower bound.
\[
\sup_t \left| \p\left(\frac{\sum x_iy_i }{\sqrt{n}}\le t\right)-\p(\frac{1}{\sqrt{n}} X^TZ\le t) \right|\ge
 O(\E \left[ 1 \wedge \left|\sqrt{\frac{\sum y_k^2}{n}}-1 \right| \right] )
\]

Now it is very likely
$
\E \left[ 1 \wedge \left|\sqrt{\frac{\sum y_k^2}{n}}-1 \right| \right]  = 1 +  O\left( \frac{1}{\sqrt{m}} \right)
$
since $\E \left[  \left(\frac{\sum {y_k}^2}{n}-1 \right)^2 \right] = O\left( \frac{1}{m} \right)$. Therefore a lower bound of $O(\frac{1}{\sqrt{m}})$ is obtained.

\smallskip
On the other hand, it is not clear whether \cref{eqn:rand_proj rate to original vector}  can be improved. In some cases, $O(\frac{1}{\sqrt{n}})$ is not necessary. For example, if we let $m\to \infty, n=1$, then
$$\frac{1}{\sqrt{m^2 1+m1^2}} X^TS^TSZ  \approx \left(\frac{1}{m} \sum_{i=1}^m S_i^2\right)XZ \to XZ$$

In the original Johnson-Lindenstrauss lemma, the number of vectors $p$ can be arbitrary ($p\ge 2$) and the error has a factor $\log p$. So far, we only discussed the case $p=2$. Moreover, we only discussed invariance of independence for random projection. To stretch the understanding to another level, we need to characterize invariance of dependent random vectors. A special case one can consider is when $X=Z$, so that we will have a quadratic form $X^TS^TSX$, which will be addressed in another future work.



\bibliography{Dependent_CLT}
\bibliographystyle{plain}

\end{document}